\documentclass[a4paper,11pt,reqno]{amsart}
\usepackage{amssymb}
\usepackage{amsmath}
\usepackage{ifthen}
\usepackage{enumitem}
\usepackage[dvips]{graphicx}
\usepackage[a4paper,margin=0.85in]{geometry}
\nonstopmode \numberwithin{equation}{section}
\usepackage{amssymb}
\usepackage{ifthen}

\usepackage{amsmath}
\usepackage[T1]{fontenc} %skandit
\usepackage{comment}

\theoremstyle{plain}
\newtheorem{thm}{Theorem}
\numberwithin{thm}{section}
\newtheorem{cor}{Corollary}
\numberwithin{cor}{section}
\newtheorem{lem}{Lemma}
\numberwithin{lem}{section}
\newtheorem{prop}{Proposition}
\newtheorem{Qsn}{Question} [section]

\newtheorem{conj}{Conjecture}

\theoremstyle{definition}
\newtheorem{defn}{Definition}[section]

\newtheorem{prob}{Problem}
\newtheorem{rem}{Remark}[section]

\newcounter{minutes}
\divide\time by 60
\newcounter{hours}
\multiply\time by 60
\addtocounter{minutes}{-\time}
\newcounter {own}
\def\theown {\thesection       .\arabic{own}}

\newenvironment{pf}[1][]{%
	\vskip 3mm
	\noindent
	\ifthenelse{\equal{#1}{}}%
	{{\slshape Proof. }}%
	{{\slshape #1.} }%
}%
{\qed\bigskip}

\theoremstyle{plain}
\newtheorem{Thm}{Theorem}

\numberwithin{equation}{section}
\def\be{\begin{equation}}
	\def\ee{\end{equation}}

\newcommand{\bee}{\begin{enumerate}}
	\newcommand{\eee}{\end{enumerate}}

\newcommand{\blem}{\begin{lem}}
	\newcommand{\elem}{\end{lem}}
\newcommand{\bthm}{\begin{thm}}
	\newcommand{\ethm}{\end{thm}}
\newcommand{\bcor}{\begin{cor}}
	\newcommand{\ecor}{\end{cor}}
\newcommand{\beg}{\begin{examp}}
	\newcommand{\eeg}{\end{examp}}
\newcommand{\begs}{\begin{examples}}
	\newcommand{\eegs}{\end{examples}}
\allowdisplaybreaks
\newcommand{\bdefn}{\begin{defn}}
	\newcommand{\edefn}{\end{defn}}

\newcommand{\bprob}{\begin{prob}}
	\newcommand{\eprob}{\end{prob}}
\newcommand{\bei}{\begin{itemize}}
	\newcommand{\eei}{\end{itemize}}

\newcommand{\bcon}{\begin{conj}}
	\newcommand{\econ}{\end{conj}}
\newcommand{\bcons}{\begin{conjs}}
	\newcommand{\econs}{\end{conjs}}
\newcommand{\bprop}{\begin{prop}}
	\newcommand{\eprop}{\end{prop}}
\newcommand{\br}{\begin{rem}}
	\newcommand{\er}{\end{rem}}
\newcommand{\brs}{\begin{rems}}
	\newcommand{\ers}{\end{rems}}
\newcommand{\bo}{\begin{obser}}
	\newcommand{\eo}{\end{obser}}
\newcommand{\bos}{\begin{obsers}}
	\newcommand{\eos}{\end{obsers}}
\newcommand{\bpf}{\begin{pf}}
	\newcommand{\epf}{\end{pf}}
\newcommand{\ba}{\begin{array}}
	\newcommand{\ea}{\end{array}}
\newcommand{\beq}{\begin{eqnarray}}
	\newcommand{\beqq}{\begin{eqnarray*}}
		\newcommand{\eeq}{\end{eqnarray}}
	\newcommand{\eeqq}{\end{eqnarray*}}

\begin{document}
	\title{Gabriel's and Frazer's problems for weighted Bergman spaces and their applications}

	\author{Himadri Halder}
	\address{Himadri Halder,
		Department of Mathematics and Computing,
		Indian Institute of Technology (ISM) Dhanbad,
		Dhanbad-826004, Jharkhand, India.}
	\email{himadrihalder@iitism.ac.in}
	
	\author{Rohit Kumar}
	\address{Rohit Kumar,
		Department of Mathematics,
		Indraprastha Institute of Information Technology Delhi,
		New Delhi-110020, India.}
	\email{rohitk12798@gmail.com}

	\subjclass[{AMS} Subject Classification:]{Primary 31A05, 30H20, 30H05}
	
	%\def\thefootnote{}
	%\footnotetext{ {\tiny File:~\jobname.tex,
			%		printed: \number\year-\number\month-\number\day,
			%		\thehours.\ifnum\theminutes<10{0}\fi\theminutes }
		%} \makeatletter\def\thefootnote{\@arabic\c@footnote}\makeatother

	\begin{abstract}
		In this paper, we investigate Gabriel's and Frazer's problems for analytic and complex-valued harmonic weighted Bergman spaces. More precisely, we establish weighted integral inequalities of the form
		\[
		\int_C |f(z)|^p(1-|z|^2)^{\alpha+1}\,|dz|
		\leq
		K_{p,\alpha,C}
		\int_{\mathbb D}|f(z)|^p(1-|z|^2)^\alpha\,dA(z),
		\]
		where $f$ is an analytic or complex-valued harmonic function on the unit disk $\mathbb D$ and $C$ is an arbitrary convex curve contained in $\mathbb{D}$. The corresponding problem was first studied by Gabriel [Proc. Lond. Math. Soc. 28 (1928), 121--127] for analytic Hardy spaces, where the inequality holds for every $0<p<\infty$. In contrast, the harmonic Hardy space analogue was recently shown to fail whenever $0<p\le1$. We prove that this phenomenon does not occur in the weighted harmonic Bergman setting by establishing Gabriel's inequality throughout the full range $0<p<\infty$.
		We further study Frazer's problem for circles and for the union of two intersecting diameters. As important applications of our main results, we derive Gabriel-type and Frazer-type inequalities for the analytic and harmonic M\"obius invariant spaces $Q(n,p,\alpha)$ and $Q_h(n,p,\alpha)$.
	\end{abstract}
	
	\maketitle
	\pagestyle{myheadings}
	\markboth{Himadri Halder and  Rohit Kumar}{Gabriel's and Frazer's problems for analytic and harmonic weighted Bergman spaces}
	
	\section{\textbf{Introduction and Main results}}
	
	\subsection{Gabriel's and Frazer's Problems in Analytic Hardy Spaces}
	Let $\mathbb{D} = \{z \in \mathbb{C} : |z| < 1\}$ denote the open unit disk, and let $\mathbb{T} = \{z \in \mathbb{C} : |z| = 1\}$ be the unit circle. For a function $f$ analytic in $\mathbb{D}$ and $0 < p < \infty$, the integral mean $M_p(r, f)$ is defined as
	\[
	M_p(r, f) = \left( \frac{1}{2\pi} \int_0^{2\pi} |f(re^{i\theta})|^p \, d\theta \right)^{\frac{1}{p}}, \quad 0 \leq r < 1.
	\]
	The function $f$ is said to belong to the classical Hardy space $H^p$ if $M_p(r, f)$ remains bounded as $r \to 1^-$. Since every function in $H^p$ possesses a radial limit almost everywhere on $\mathbb{T}$, the Hardy norm is naturally expressed in terms of boundary integrals. This boundary integral structure has led to the development of numerous inequalities comparing the size of analytic functions over different geometric subsets of the unit disk. For a comprehensive account of Hardy spaces, we refer the reader to the monograph by Duren \cite{Duren-Hp-book}.
	\vspace{1mm}
	
	Integral inequalities have long played a central role in complex analysis, geometric function theory, and harmonic analysis. Among these, inequalities comparing the integral of an analytic function over different subsets of the unit disk have attracted sustained interest because of their close connections with growth estimates, extremal problems, coefficient estimates, and geometric properties of analytic mappings. One of the earliest and most influential results in this direction is the classical Riesz--Fejér inequality, which asserts that for every $f\in H^p$, $0<p<\infty$,
	$$
	\int_{-1}^{1}|f(x)|^{p}\,dx
	\le
	\frac12\int_{0}^{2\pi}|f(e^{i\theta})|^{p}\,d\theta,
	$$
	where the constant $1/2$ is sharp. Besides its intrinsic importance, the Riesz--Fejér inequality has become one of the fundamental tools in the study of Hardy spaces and has inspired numerous subsequent developments. 	  
	\vspace{2mm}
	
	The influence of the Riesz--Fej\'er theorem extends far beyond its original formulation. Beckenbach \cite{Beckenbach-1938} showed that the inequality remains valid when $|f|^p$ is replaced by an arbitrary positive function whose logarithm is subharmonic, thereby revealing that the phenomenon is fundamentally connected with subharmonicity rather than analyticity alone.
	\vspace{1mm}
	
	Motivated by these developments, Gabriel \cite{Gabriel-1928} investigated geometric extensions of the Riesz--Fej\'er inequality by replacing the diameter with an arbitrary convex curve contained in the unit disk. His celebrated theorem, now regarded as the origin of Gabriel's problem, is stated below (see also the survey by Granados \cite{Granados-1999}).
	
	\begin{Thm}[Gabriel, \cite{Gabriel-1928}] \label{thm-A}
		Let $f \in H^p$ for $p > 0$. If $C$ is any convex curve in $\mathbb{D}$, then
		\[
		\int_C |f(z)|^p \, |dz| \leq 2 \int_{\mathbb{T}} |f(z)|^p \, |dz|.
		\]
		The constant $2$ is the best possible.
	\end{Thm}
	
	The theory was substantially enriched by Gabriel and Frazer through a series of influential papers devoted to special geometric configurations (see \cite{Gabriel-1928-a}-\cite{Gabriel-1928-c}, \cite{Frazer-1932}-\cite{Frazer-1945}). A particularly noteworthy achievement is Frazer's theorem for circular curves, which establishes that the optimal constant admits a dramatic improvement when the underlying convex curve is a circle \cite{Frazer-1945}.
	
	\begin{Thm}[Frazer, \cite{Frazer-1945}] \label{thm-B}
		If $f \in H^p$ for some $p > 0$ and $C$ is a circle in $\mathbb{D}$, then
		\[
		\int_C |f(z)|^p \, |dz| \leq \int_{\mathbb{T}} |f(z)|^p \, |dz|.
		\]
	\end{Thm}
	
	Prior to the circle case, Frazer also investigated the integration of analytic functions along a pair of intersecting diameters, producing a fascinating genralization of the Riesz-Fej\'er theorem. 
	
	\begin{Thm}[Frazer, \cite{Frazer-1934_regular}] \label{thm-C}
		If $f \in H^p$ for $p > 0$, and $D_0, D_1$ are any two diameters of $\mathbb{T}$ intersecting at an acute angle $\theta$ ($0 \leq \theta \leq \pi/2$), then
		\[
		\int_{D_0 + D_1} |f(z)|^p \, |dz| \leq \frac{1}{\sin\frac{\theta}{2} + \cos\frac{\theta}{2}} \int_{\mathbb{T}} |f(z)|^p \, |dz|.
		\]
	\end{Thm}
	Observe that when $\theta=0$, the two diameters coincide, and the inequality reduces to the classical Riesz--Fej\'er inequality. This result further illustrates how geometric considerations naturally lead from the Riesz--Fej\'er inequality toward Gabriel's problem.
	
	\subsection{Gabriel's and Frazer's Problems in Harmonic Hardy Spaces}
	The remarkable success of Gabriel's problem in analytic Hardy spaces naturally led researchers to investigate whether analogous inequalities remain valid for harmonic mappings. This question is particularly significant because complex-valued harmonic mappings arise naturally in geometric function theory, quasiconformal mappings, minimal surfaces, nonlinear elasticity, and various boundary value problems for elliptic partial differential equations. Extending classical analytic inequalities to harmonic mappings has therefore become one of the central themes of modern function theory.
	\vspace{1mm}
	
	A complex-valued function $f=u+iv$ is said to be harmonic in $\mathbb{D}$ if both $u$ and $v$ are real-valued harmonic functions. Every harmonic mapping admits the canonical decomposition $f=h+\overline{g}$, where $h$ and $g$ are analytic in $\mathbb{D}$. For $0<p<\infty$, the harmonic Hardy space $h^p$ consists of all harmonic mappings $f$ satisfying \[ \|f\|_{h^p} := \sup_{0<r<1} M_p(r,f) <\infty. \] As in the analytic case, every function in $h^p$ possesses radial boundary values almost everywhere on $\mathbb{T}$.
	\vspace{1mm}
	
	Unlike analytic functions, harmonic mappings lack several key properties that are fundamental to the proofs of the classical Riesz--Fej\'er and Gabriel's theorem. In particular, for $0<p<1$, the function $|f|^p$ is generally not subharmonic. As a result, the classical methods based on the subharmonicity of analytic functions cannot be directly adapted to the harmonic setting. Therefore, extending Riesz--Fej\'er and Gabriel's result from analytic to harmonic Hardy spaces presents substantial analytical difficulties. 
	\vspace{1mm}
	
	A harmonic analogue of the classical Riesz--Fej\'er inequality for $1<p\leq 2$ was established by Kayumov, Ponnusamy, and Kaliraj \cite{Kayumov-Ponnusamy-Kaliraj-2020}.
	\begin{Thm} \cite{Kayumov-Ponnusamy-Kaliraj-2020} \label{thm-D}
		If $f\in h^p$ for $1<p\le2$, then \[ \int_{-1}^{1}|f(x)|^p\,dx \le \frac12\sec^p\!\left(\frac{\pi}{2p}\right)\, \int_{\mathbb{T}}|f(z)|^p\,|dz|. \] 
		The inequality is sharp. 
	\end{Thm} 
	Melentijevi\'c and Bo\v{z}in \cite{Melentijevic-Bozin-2021} later extended this result to the range $p\ge2$, showing that the same inequality holds with the sharp constant $1/2$ on the right side. 
	\vspace{1mm}
	
	Building upon these results, Das and Kaliraj \cite{Das-Kaliraj-2022} established the harmonic analogue of Frazer's theorem concerning the union of two intersecting diameters. 
	\begin{Thm}\cite{Das-Kaliraj-2022} \label{thm-E}
		Let $f\in h^p$, where $p>1$. If $D_0$ and $D_1$ are two diameters of $\mathbb{T}$ intersecting at an acute angle $\theta$, then \[ \int_{D_0\cup D_1}|f(z)|^p\,|dz| \le A_p(\theta) \int_{\mathbb{T}}|f(z)|^p\,|dz|, \] where \[
		A_p(\theta):=
		\begin{cases}
			\displaystyle
			\sec^{p}\!\left(\frac{\pi}{2p}\right)
			\frac{1}{\sin\frac{\theta}{2}+\cos\frac{\theta}{2}},
			& 1<p\le2,\\[3mm]
			\displaystyle
			\frac{2}{\sin\frac{\theta}{2}+\cos\frac{\theta}{2}},
			& p\ge2.
		\end{cases}
		\]
	\end{Thm} 
	Significant progress on higher-dimensional Riesz--Fej\'er inequalities has recently been made by Chen and Hamada \cite{Chen-Hamada-MZ-2023}, who established Fejér--Riesz type inequalities for pluriharmonic functions in $\mathbb{R}^{2n}$.
	\vspace{1mm}
	
	More recently, Das \cite{Das-2025} considered Gabriel's problem in the harmonic Hardy space and proved the following theorem.
	\begin{Thm}\cite{Das-2025}\label{thm-F} 
		Let $f\in h^p$ for some $p>1$, and let $C$ be any convex curve in $\mathbb{D}$. Then \[ \int_C|f(z)|^p\,|dz| \le 4 \int_{\mathbb{T}}|f(z)|^p\,|dz|, \qquad p\ge2, \] and \[ \int_C|f(z)|^p\,|dz| \le 2\sec^p\!\left(\frac{\pi}{2p}\right) \int_{\mathbb{T}}|f(z)|^p\,|dz|, \qquad 1<p<2. \] 
	\end{Thm} 
	Theorem \ref{thm-F} completes the extension of the classical Gabriel's theorem from analytic Hardy spaces to harmonic Hardy spaces. 
	
	\subsection{Motivation and Main Results in Weighted Bergman Spaces} \label{subsec-1.3}
	The results discussed in the previous subsections are established in the setting of Hardy spaces, where the underlying norm is determined by boundary values on the unit circle. Consequently, the corresponding proofs rely heavily on boundary integral representations, the Poisson kernel, and the subharmonicity of functions. In contrast, Bergman spaces are defined by interior weighted area integrals rather than boundary values. It is therefore natural to ask whether analogous Gabriel-type inequalities continue to hold in the Bergman space setting.
	
	Let $dA(z)$ denote the normalized area measure on $\mathbb{D}$, given by
	\[
	dA(z)=\frac{dx\,dy}{\pi}
	=\frac{1}{\pi}r\,dr\,d\theta,
	\qquad
	z=x+iy=re^{i\theta}.
	\]
	For $-1<\alpha<\infty$, the weighted area measure is defined by
	\[
	dA_\alpha(z)
	=
	(\alpha+1)(1-|z|^2)^\alpha\,dA(z).
	\]
	
	For $0<p<\infty$, the weighted analytic Bergman space $A_\alpha^p$ consists of all analytic functions $f$ on $\mathbb D$ satisfying
	\[
	\|f\|_{A_\alpha^p}
	=
	\left(
	\int_{\mathbb D}|f(z)|^p\,dA_\alpha(z)
	\right)^{1/p}
	<\infty.
	\]
	When $\alpha=0$, this reduces to the classical Bergman space $A^p$. Likewise, the weighted harmonic Bergman space $a_\alpha^p$ consists of all complex-valued harmonic functions $f$ on $\mathbb D$ for which
	\[
	\|f\|_{a_\alpha^p}
	=
	\left(
	\int_{\mathbb D}|f(z)|^p\,dA_\alpha(z)
	\right)^{1/p}
	<\infty.
	\]
	For $\alpha=0$, we simply write $a_\alpha^p=a^p$. Moreover, it is well known that $H^p\subset A^p$ and $h^p\subset a^p$.
	For a comprehensive account of Bergman spaces, we refer the reader to the monographs of Duren and Schuster \cite{Duren-Schuster-AMS-2004}, Zhu \cite{Zhu-OTFS-2007}, and Hedenmalm, Korenblum, and Zhu \cite{Hedenmalm-Korenblum-Zhu-GTM}.
	
	The transition from Hardy spaces to Bergman spaces has attracted considerable attention over the past two decades. A significant contribution in this direction was made by Zhu \cite{Zhu-Monthly-2004}, who established a systematic correspondence between Hardy and weighted Bergman spaces. As an application of this correspondence, he obtained the following weighted Bergman analogue of the classical Riesz--Fej\'er inequality.
	
	\begin{Thm}\cite[Theorem 4]{Zhu-Monthly-2004}
		\label{thm-G}
		Let $0<p<\infty$ and $\alpha>-1$. If
		$f\in A_\alpha^p$, then
		\[
		\int_{-1}^{1}
		(1-|x|)^{\alpha+1}|f(x)|^p\,dx
		\le
		\pi
		\int_{\mathbb D}|f(z)|^p\,dA_\alpha(z).
		\]
	\end{Thm}
	
	Motivated by Zhu's theorem, the present authors \cite{Halder-Kumar-P1-2026} recently established the corresponding Riesz--Fej\'er inequality for weighted harmonic Bergman spaces.
	
	\begin{Thm}\cite[Theorem 1.2]{Halder-Kumar-P1-2026}
		\label{thm-I}
		Let $1<p<\infty$ and $-1<\alpha<\infty$. If
		$f=h+\overline g\in a_\alpha^p$, then
		\[
		\int_{-1}^{1}
		(1-|x|)^{\alpha+1}|f(x)|^p\,dx
		\le
		\pi
		\sec^p\!\left(\frac{\pi}{2p}\right)
		\int_{\mathbb D}|f(z)|^p\,dA_\alpha(z).
		\]
	\end{Thm}
	
	The above results demonstrate that the classical Riesz--Fej\'er inequality admits natural extensions from Hardy spaces to weighted Bergman spaces. This naturally leads to the following question.
	
	\begin{Qsn}
		Can Gabriel's problem be developed in the settings of weighted analytic and harmonic Bergman spaces? More precisely, do the classical results of Gabriel and Frazer, namely Theorems~\ref{thm-A}, \ref{thm-B}, \ref{thm-C}, \ref{thm-E}, and \ref{thm-F}, admit analogues in these spaces?
	\end{Qsn}
	
	The primary objective of the present paper is to answer this question affirmatively. More precisely, we develop a Bergman-space theory of Gabriel's and Frazer's problem for both weighted analytic and harmonic function spaces.
	\vspace{1mm}
	
	Our first main result establishes a Gabriel-type inequality for weighted analytic Bergman spaces, which serves as the Bergman-space analogue of the classical Gabriel's theorem, namely Theorem \ref{thm-A}.
	\begin{thm}\label{Theorem-2.1}
		Let $f \in A_\alpha^p$, $p > 0$. If $C$ is any convex curve in $\mathbb{D}$, then$$\int_C |f(z)|^p (1-|z|^2)^{1+\alpha} |dz| \leq D_{C,\alpha} \int_{\mathbb{D}} |f(z)|^p (1-|z|^2)^\alpha dA(z).$$
		Here, $D_{C,\alpha} = \frac{4\pi}{(1-R^2)^{1+\alpha}}$, $-1 < \alpha < \infty$ and $R=\sup_{z \in C} |z|$.
	\end{thm} 
	The preceding theorem immediately yields the following Gabriel-type inequality for the standard Bergman space $A^p$.
	\begin{cor}
		Let $f \in A^p$, $p > 0$. If $C$ is any convex curve in $\mathbb{D}$, then$$\int_C |f(z)|^p (1-|z|^2) |dz| \leq D_{C} \int_{\mathbb{D}} |f(z)|^p (1-|z|^2) dA(z).$$
		Here, $D_{C} = \frac{4\pi}{(1-R^2)}$ and $R=\sup_{z \in C} |z|$.
	\end{cor}
	We now investigate Gabriel's problem in the setting of weighted harmonic Bergman spaces. As an application of Theorem \ref{Theorem-2.1}, we derive the weighted harmonic Bergman analogue of Theorem \ref{thm-F}.
	\begin{thm}\label{Theorem-2.2}
		Let $f \in a_\alpha^p$, $0 < p < \infty$, and let $C$ be any convex curve in $\mathbb{D}$. Then$$\int_C |f(z)|^p(1-|z|^2)^{1+\alpha}|dz| \leq K_{p,\alpha,C} \int_{\mathbb{D}}|f(z)|^p(1-|z|^2)^\alpha dA(z),$$where $K_{p,\alpha,C}$ is a positive constant depending only on $p$, $\alpha$, and $C$.
	\end{thm}
	\begin{rem}
		A noteworthy feature of our result is the contrast with the Hardy space setting. While Gabriel's theorem, namely Theorem \ref{thm-A}, is valid for analytic Hardy spaces for all $0<p<\infty$, its harmonic counterpart fails whenever $0<p\le1$ (see \cite{Das-2025}). Theorem \ref{Theorem-2.2} demonstrates that this phenomenon does not occur in weighted harmonic Bergman spaces, where Gabriel's result remains valid throughout the full range $0<p<\infty$.
	\end{rem}
	
	Next, following the line of investigation initiated by Frazer \cite{Frazer-1945}, we consider the highly structured case in which the convex curve is the circle $	C_r=\{z\in\mathbb D:|z|=r\}$. Since the weight is constant on $C_r$, this enables us to derive a more refined bound, thereby obtaining a weighted harmonic Bergman analogue of Theorem \ref{thm-B}.
	\begin{thm}\label{Theorem-2.3}
		Let $0<p<\infty$, $\alpha>-1$, and let $f\in a_\alpha^p$. Then, for every $0<r<1$,
		\[
		\int_{C_r}|f(z)|^p(1-|z|^2)^{1+\alpha} |dz|
		\leq
		2\, \pi \,r(\alpha+1)\,
		\int_{\mathbb D}|f(z)|^p(1-|z|^2)^\alpha\,dA(z),
		\]
		where $	C_r=\{z\in\mathbb D:|z|=r\}$.
	\end{thm}
	
	The sharp inequalities of Kalaj \cite{Kalaj-2019} for harmonic Hardy and Bergman spaces establish a fundamental quantitative relationship between the natural Euclidean quantity $(|g|^{2}+|h|^{2})^{1/2}$ and the harmonic mapping $f=g+\overline{h}$. These inequalities extend the classical Riesz conjugate theory for analytic functions to complex-valued harmonic mappings and have important applications to the boundedness of the Hilbert transform, Riesz-type inequalities, and isoperimetric estimates. They also provide a powerful tool for relating the analytic and co-analytic parts of a harmonic mapping while preserving the sharp constants.
	\vspace{1mm}
	
	The presence of the weight $(1-|z|^2)^\alpha$ in the weighted harmonic Bergman spaces makes the extension of Kalaj's inequalities highly nontrivial. It is therefore natural to ask whether analogous sharp inequalities of \cite[Theorem 2.1 and Corollary 2.9]{Kalaj-2019} remain valid in the weighted setting.
	
	In this paper, we answer this question affirmatively by proving weighted harmonic Bergman analogues of these inequalities. Since the classical Bergman space corresponds to the case $\alpha=0$, our results recover the known Bergman inequalities as a special case while providing a unified framework for weighted harmonic Bergman spaces. 
	
	Beyond their independent interest, these inequalities serve as a key ingredient in proving a weighted harmonic Bergman analogue of Frazer's theorem. Thus, our following result not only generalizes Kalaj's inequalities to the weighted setting but also demonstrate their effectiveness by solving a classical extremal problem in weighted harmonic Bergman spaces.
	\begin{lem}\label{Lemma-2.1}
		Let $1 < p < \infty$. Assume that $f = h + \bar{g}$ belongs to the weighted harmonic Bergman space $a_\alpha^p$, where $h$ and $g$ are analytic functions on $\mathbb{D}$. If $\Re(g(0)h(0)) \ge 0$, then$$\left( \int_{\mathbb{D}} (|h|^2 + |g|^2)^{p/2} dA_\alpha \right)^{1/p} \le \frac{1}{(1-|\cos(\pi/p)|)^{1/2}} \left( \int_{\mathbb{D}} |h+\bar{g}|^p dA_\alpha \right)^{1/p}$$
	\end{lem}
	By making use of this lemma, we now establish weighted harmonic Bergman analogues of the results stated in Theorems \ref{thm-C} and \ref{thm-E} for the union of two intersecting diameters of circle.
	\begin{thm}\label{Theorem-2.4}
		If $f \in a_{\alpha}^p$ for some $p > 1$, $\alpha\geq0$ and $f(0)=0$, then the following inequality holds:$$\int_{D_0+D_1} |f(z)|^p (1-|z|^2)^{1+\alpha} |dz| \le A_p(\theta, \alpha) \|f\|^p_{a_{\alpha}^p},$$where $D_0, D_1$ are two diameters of $\mathbb{T}$, $\theta$ is the acute angle between them, and $$A_p(\theta, \alpha)= \frac{2^{1+\alpha}}{\sin\frac{\theta}{2} + \cos\frac{\theta}{2}} \frac{2^{p/2}}{(1-|\cos(\frac{\pi}{p})|)^{p/2}}.$$
	\end{thm}
	\begin{rem}
		Specializing to the case $\alpha=0$ in Theorems \ref{Theorem-2.2}, \ref{Theorem-2.3}, and \ref{Theorem-2.4} recovers the corresponding results for the standard harmonic Bergman space $a^p$.
	\end{rem}
	
	\section{\textbf{Application: Gabriel's and Frazer's problem for M\"obius invariant spaces}}
	M\"obius invariant spaces constitute one of the central classes of function spaces in complex analysis. Their defining feature is invariance under automorphisms of the unit disk, a property that plays a fundamental role in geometric function theory, operator theory, and conformal analysis. Several classical spaces, including the Bloch space, BMOA, Besov spaces, Dirichlet spaces, and the $Q_s$-spaces, arise naturally within this framework.
	
	To provide a unified treatment of these spaces, Zhu introduced the family of analytic M\"obius invariant spaces $Q(n,p,\alpha)$ in \cite{Zhu-IJM-2007}. Let $0<p<\infty$, $\alpha>-1$, and $n\in\mathbb N$. An analytic function $f$ belongs to $Q(n,p,\alpha)$ if
	\begin{equation}\label{Qanalytic}
		\|f\|_{Q(n,p,\alpha)}^{p}
		:=
		\sup_{a\in\mathbb D}
		\int_{\mathbb D}
		\left|(f\circ\sigma_a)^{(n)}(z)\right|^{p}
		(1-|z|^{2})^{\alpha}\,dA(z)
		<\infty,
	\end{equation}
	where
	\[
	\sigma_a(z)=\frac{a-z}{1-\overline{a}z},
	\qquad a\in\mathbb D,
	\]
	denotes the M\"obius automorphism of $\mathbb D$ interchanging $0$ and $a$.
	
	Since every automorphism of $\mathbb D$ is the composition of a rotation with $\sigma_a$, the seminorm in \eqref{Qanalytic} is M\"obius invariant. Consequently,
	\[
	f\in Q(n,p,\alpha)
	\quad \mbox{is equivalent to saying}\quad
	f\circ\sigma\in Q(n,p,\alpha)
	\]
	for every $\sigma\in{\rm Aut}(\mathbb D)$, and
	\[
	\|f\circ\sigma\|_{Q(n,p,\alpha)}
	=
	\|f\|_{Q(n,p,\alpha)}.
	\]
	
	The family $Q(n,p,\alpha)$ provides many well-known analytic function spaces. When $p\ge1$, the quantity $|f(0)|+\|f\|_{Q(n,p,\alpha)}$ defines a complete norm on $Q(n,p,\alpha)$, making it a Banach space. In particular, it coincides with the Bloch space whenever $np<\alpha+1$, reduces to the Besov space when $np=\alpha+2$, and contains the classical $Q_s$-spaces and BMOA as important special cases. Furthermore, these spaces admit elegant descriptions in terms of Carleson measures and lacunary series; see \cite{Zhu-IJM-2007} for a comprehensive account.
	
	In recent years, there has been considerable interest in extending M\"obius invariant function spaces to the harmonic setting. Motivated by applications to quasiconformal mappings, minimal surfaces, nonlinear elasticity, and geometric function theory, Sun {\it et al.} \cite{Sun-Liu-Wang-PA-2026} introduced the harmonic M\"obius invariant spaces $Q_h(n,p,\alpha)$. A harmonic mapping $f=h+\overline g$ belongs to $Q_h(n,p,\alpha)$ whenever
	\begin{equation}\label{Qharmonic}
		\|f\|_{Q_h(n,p,\alpha)}^{p}
		=
		\sup_{a\in\mathbb D}
		\int_{\mathbb D}
		\left(
		\left|(h\circ\sigma_a)^{(n)}(z)\right|
		+
		\left|(g\circ\sigma_a)^{(n)}(z)\right|
		\right)^p
		(1-|z|^2)^\alpha
		\,dA(z)
		<\infty.
	\end{equation}
	
	For the first-order case $n=1$, the above seminorm admits the equivalent representation
	\[
	\|f\|_{Q_h(1,p,\alpha)}^{p}
	\approx
	\sup_{a\in\mathbb D}
	\int_{\mathbb D}
	\Lambda_{f\circ\sigma_a}(z)^p
	(1-|z|^2)^\alpha\,dA(z),
	\]
	where
	\[
	\Lambda_f(z)=|f_z(z)|+|f_{\overline z}(z)|.
	\]
	For two nonnegative quantities $A$ and $B$, the notation $A\lesssim B$ means that there exists a positive constant $C$ such that $A\le CB$. We write $A\gtrsim B$ if $B\lesssim A$, and $A\approx B$ whenever both estimates hold.
	\vspace{1mm}
	
	Since
	\[
	\Lambda_f(z)
	\le |\nabla f(z)|
	\le \sqrt2\,\Lambda_f(z),
	\]
	the seminorm may equivalently be expressed in terms of the Euclidean gradient, where $|\nabla f(z)|:=(|f_x(z)|^2+|f_y(z)|^2)^{1/2}$.
	
	The spaces $Q_h(n,p,\alpha)$ retain the M\"obius invariance of their analytic counterparts while naturally incorporating harmonic mappings. They provide a common setting for harmonic Bloch spaces, harmonic $Q_s$-spaces, harmonic BMO-type spaces, and several classes of quasiregular harmonic mappings.
	
	Although the theory of harmonic M\"obius invariant spaces has witnessed rapid progress in recent years, many classical extremal problems remain unexplored in this setting. In particular, Gabriel's and Frazer's problems have not previously been investigated for either the analytic spaces $Q(n,p,\alpha)$ or their harmonic counterparts $Q_h(n,p,\alpha)$. 
	\vspace{1mm}
	
	Motivated by the classical Riesz conjugate theorem, Sun {\it et al.} recently established its harmonic analogue for the first-order M\"obius invariant space $Q_h(1,p,\alpha)$ in \cite[Theorem 3.1]{Sun-Liu-Wang-PA-2026} whenever $1+\alpha <p<2+\alpha$. 
	\vspace{1mm}
	
	We now recall the following weighted Riesz--Fej\'er inequalities obtained by the present authors in \cite{Halder-Kumar-P1-2026} for the aforementioned classes.
	\begin{Thm} \label{thm-2.2} \cite[Theorem 2.2]{Halder-Kumar-P1-2026}
		Let $ 1< p < \infty$ and $\alpha>np-1$.
		\begin{enumerate}[label=(\roman*)]
			\item If $f $ is a complex-valued harmonic function belonging to the space $Q_h(n,p,\alpha)$, then the following integral inequality holds:
			$$
			\int_{-1}^{1} (1 - |x|)^{\alpha+1-np} |f(x)|^p \, dx \leq  \pi \, \sec^{p}\left(\dfrac{\pi}{2p}\right)\, \|f\|^p_{a^p_{\alpha-np}}.
			$$
			\item If $f \in Q(n,p,\alpha)$, then the following integral inequality holds:
			$$
			\int_{-1}^{1} (1 - |x|)^{\alpha+1-np} |f(x)|^p \, dx \leq  \pi \,  \|f\|^p_{A^p_{\alpha-np}}.
			$$
		\end{enumerate}
	\end{Thm}
	
	To the best of our knowledge, no analogue of Gabriel's or Frazer's problem has been established for the M\"obius invariant spaces $Q(n,p,\alpha)$ and $Q_h(n,p,\alpha)$. The purpose of this section is to fill this gap by deriving the corresponding inequalities as applications of our main results.
	\vspace{1mm}
	
	As a preliminary ingredient, we recall the following embedding theorem established in \cite[Lemma 2.1]{Halder-Kumar-P1-2026}, which connects M\"obius invariant spaces with weighted Bergman spaces.
	\begin{lem} \label{lem-2.1} \cite[Lemma 2.1]{Halder-Kumar-P1-2026}
		Let $0 < p < \infty$ and $\alpha>np-1$. Then $Q(n,p,\alpha)\subset A_{\alpha-np}^p$ and $Q_h(n,p,\alpha)\subset a_{\alpha-np}^p$.
	\end{lem}
	We are now in a position to combine the embedding in Lemma \ref{lem-2.1} with the main results established in the previous section to derive Gabriel-type and Frazer-type inequalities for M\"obius invariant spaces.
	\vspace{1mm}
	
	Our first application is a Gabriel-type inequality for the analytic M\"obius invariant space $Q(n,p,\alpha)$, which follows immediately from Theorem \ref{Theorem-2.1} and Lemma \ref{lem-2.1}.
	\begin{cor}\label{Theorem-2.1-a}
		Let $f \in Q(n,p,\alpha)$, where $ 1< p < \infty$ and $\alpha>np-1$. If $C$ is any convex curve in $\mathbb{D}$, then$$\int_C |f(z)|^p (1-|z|^2)^{\alpha+1-np} |dz| \leq D_{C,\alpha} \int_{\mathbb{D}} |f(z)|^p (1-|z|^2)^{\alpha-np} dA(z).$$
		Here, $D_{C,\alpha} = \frac{4\pi}{(1-R^2)^{\alpha+1-np}}$, $-1 < \alpha < \infty$ and $R=\sup_{z \in C} |z|$.
	\end{cor}
	
	Following the analytic case, an application of Theorem \ref{Theorem-2.2} together with Lemma \ref{lem-2.1} yields the corresponding Gabriel-type inequality for the space $Q_h(n,p,\alpha)$.
	
	\begin{cor}\label{Theorem-2.2-a}
		Let $f \in Q_h(n,p,\alpha)$, $ 1< p < \infty$ and $\alpha>np-1$. Suppose $C$ is any convex curve in $\mathbb{D}$. Then$$\int_C |f(z)|^p(1-|z|^2)^{\alpha+1-np}|dz| \leq K_{p,\alpha,C} \int_{\mathbb{D}}|f(z)|^p(1-|z|^2)^{\alpha-np} dA(z),$$where $K_{p,\alpha,C}$ is a positive constant depending only on $p$, $\alpha$, and $C$.
	\end{cor}
	
	As an immediate consequence of the circular version of Gabriel's inequality, Theorem \ref{Theorem-2.3}, together with Lemma \ref{lem-2.1}, yields the following Frazer-type inequality for the harmonic M\"obius invariant space $Q_h(n,p,\alpha)$.
	\begin{cor}\label{Theorem-2.3-a}
		Let $f \in Q_h(n,p,\alpha)$ with  $ 1< p < \infty$ and $\alpha>np-1$. Then, for every $0<r<1$,
		\[
		\int_{C_r}|f(z)|^p(1-|z|^2)^{\alpha+1-np} |dz|
		\leq
		2\, \pi \,r(\alpha+1-np)\,
		\int_{\mathbb D}|f(z)|^p(1-|z|^2)^{\alpha-np}\,dA(z),
		\]
		where $	C_r=\{z\in\mathbb D:|z|=r\}$.
	\end{cor}
	Our final application follows from Theorem \ref{Theorem-2.4} together with Lemma \ref{lem-2.1}. It establishes analogues of Theorems \ref{thm-C} and \ref{thm-E} for the space $Q_h(n,p,\alpha)$ on the union of two intersecting diameters.
	\begin{cor}\label{Theorem-2.4-a}
		Let $f \in Q_h(n,p,\alpha)$ for $ 1< p < \infty$ and $\alpha>np-1$, and $f(0)=0$. Then the following inequality holds:$$\int_{D_0+D_1} |f(z)|^p (1-|z|^2)^{1+\alpha-np} |dz| \le B_p(\theta, \alpha) \|f\|^p_{a_{\alpha-np}^p},$$
		where $D_0, D_1$ are two diameters of $\mathbb{T}$, $\theta$ is the acute angle between them, and $$B_p(\theta, \alpha)= \frac{2^{1+\alpha-np}}{\sin\frac{\theta}{2} + \cos\frac{\theta}{2}} \frac{2^{p/2}}{(1-|\cos(\frac{\pi}{p})|)^{p/2}}.$$
	\end{cor}
	Each of the above corollaries follows immediately by combining the corresponding results established in Subsection \ref{subsec-1.3} with the embedding stated in Lemma \ref{lem-2.1}. We therefore omit the proofs.
	\section{\textbf{Proofs of the Main Results}}
	
	\begin{pf} [{\bf Proof of Theorem \ref{Theorem-2.1}}]
		Let $f\in A_\alpha^p$, where $p>0$, and for each $0<r<1$ define
		$
		f_r(z):=f(rz).
		$
		Since $f$ is analytic in $\mathbb D$, it follows that $f_r\in H^p$.
		
		Let
		$
		R=\sup_{z\in C}|z|
		$. Since $C$ is compact, the supremum is finite and $R<1$. For every $r\in(R,1)$, define
		$$
		C_r=\{\xi\in\mathbb D:r\xi\in C\}.
		$$
		Since $C$ is convex, it follows that $C_r$ is also a convex curve contained in $\mathbb D$.
		
		Applying Theorem~\ref{thm-A} to the function $f_r$, we obtain
		$$
		\int_{C_r}|f_r(z)|^p\,|dz|
		\le
		2\int_{|z|=1}|f_r(z)|^p\,|dz|.
		$$
		Multiplying both sides by $r$ and making the change of variables $w=rz$, we obtain
		$$
		\int_C|f(w)|^p\,|dw|
		\le
		2r\int_{|z|=1}|f(rz)|^p\,|dz|.
		$$
		Hence,
		\begin{align}
			\int_C|f(z)|^p\,|dz|
			&\le
			2r\int_0^{2\pi}|f(re^{i\theta})|^p\,d\theta\notag\\
			&\le
			2\int_0^{2\pi}|f(re^{i\theta})|^p\,d\theta,
			\label{eq:curve-estimate}
		\end{align}
		since $r<1$.
		
		Since \eqref{eq:curve-estimate} holds for every $R<r<1$, multiplying both sides by $(1-r^2)^\alpha r$ and integrating from $R$ to $1$, we obtain
		$$
		\left(\int_C|f(z)|^p\,|dz|\right)
		\int_R^1(1-r^2)^\alpha r\,dr
		\le
		2\int_R^1
		\left(
		\int_0^{2\pi}|f(re^{i\theta})|^p\,d\theta
		\right)
		(1-r^2)^\alpha r\,dr.
		$$
		Since
		$$
		dA_\alpha(z)
		=(1+\alpha)(1-|z|^2)^\alpha\,dA(z),
		\qquad
		dA(z)=\frac{1}{\pi}\,dx\,dy,
		$$
		we have
		$$
		\|f\|_{A_\alpha^p}^p
		=
		\frac{1+\alpha}{\pi}
		\int_0^1
		\int_0^{2\pi}
		|f(re^{i\theta})|^p
		(1-r^2)^\alpha
		r\,d\theta\,dr.
		$$
		Therefore,
		$$
		\int_0^1
		\int_0^{2\pi}
		|f(re^{i\theta})|^p
		(1-r^2)^\alpha
		r\,d\theta\,dr
		=
		\frac{\pi}{1+\alpha}
		\|f\|_{A_\alpha^p}^p.
		$$
		Consequently,
		\begin{align}
			\int_C|f(z)|^p\,|dz|
			&\le
			\frac{2}{\displaystyle\int_R^1(1-r^2)^\alpha r\,dr}
			\cdot
			\frac{\pi}{1+\alpha}
			\|f\|_{A_\alpha^p}^p\notag\\
			&=
			\frac{2\pi}
			{\displaystyle\frac{(1-R^2)^{\alpha+1}}{2}}
			\|f\|_{A_\alpha^p}^p\notag\\
			&=
			\frac{4\pi}{(1-R^2)^{\alpha+1}}
			\|f\|_{A_\alpha^p}^p,
			\label{eq:curve-final}
		\end{align}
		where we have used
		$$
		\int_R^1(1-r^2)^\alpha r\,dr
		=
		\frac{(1-R^2)^{\alpha+1}}{2(1+\alpha)}.
		$$
		
		Finally, since $\alpha>-1$, we have
		$
		(1-|z|^2)^{\alpha+1}\le1
		$, and therefore
		\begin{equation}\label{eq:weight}
			\int_C|f(z)|^p(1-|z|^2)^{\alpha+1}\,|dz|
			\le
			\int_C|f(z)|^p\,|dz|.
		\end{equation}
		Combining \eqref{eq:curve-final} and \eqref{eq:weight}, we conclude that
		$$
		\int_C|f(z)|^p(1-|z|^2)^{\alpha+1}\,|dz|
		\le
		D_{C,\alpha}\|f\|_{A_\alpha^p}^p,
		$$
		where
		$$
		D_{C,\alpha}
		=
		\frac{4\pi}{(1-R^2)^{\alpha+1}}.
		$$
		This completes the proof.
	\end{pf}

	\begin{pf} [{\bf Proof of Theorem \ref{Theorem-2.2}}]
		Let
		\[
		f=u+iv=h+\overline{g},
		\]
		where $h$ and $g$ are analytic in $\mathbb{D}$. Define
		$F=h+g$. Then $F$ is analytic and $u=\Re(F)$. Since $|u(z)|\leq |f(z)|$ for $ z\in\mathbb{D}$, it follows immediately that
		\[
		\|u\|_{a_\alpha^p}\leq \|f\|_{a_\alpha^p}< \infty.
		\]
		Invoking the analytic Bergman estimate \cite[Corollary 6]{Pelaez-Rattya-AMP-2020} to the function $F$, we conclude that there exists a constant $c_{p,\alpha}>0$, depending only on $p$ and $\alpha$, such that
		\[
		\|F\|_{A_\alpha^p}\leq c_{p,\alpha}\,\|u\|_{a_\alpha^p}< \infty.
		\]
		Therefore,
		\begin{equation}\label{-1c}
			\int_{\mathbb D}|F(z)|^p\,dA_\alpha(z)
			\leq
			c_{p,\alpha}^p
			\int_{\mathbb D}|u(z)|^p\,dA_\alpha(z)\leq c_{p,\alpha}^p \int_{\mathbb D}|f(z)|^p\,dA_\alpha(z) .
		\end{equation}
		
		Since $F=h+g\in A_\alpha^p$, applying Theorem \ref{Theorem-2.1} to $F$ gives
		\[
		\int_C |F(z)|^p(1-|z|^2)^{1+\alpha}|dz|
		\leq
		D_{C,\alpha}
		\int_{\mathbb D}|F(z)|^p\,dA_\alpha(z),
		\]
		where
		\[
		D_{C,\alpha}
		=
		\frac{4\pi}
		{(1-R^2)^{1+\alpha}}
		\qquad \mbox{and} \qquad
		R=\sup_{z\in C}|z|.
		\]
		Combining the above inequality with \eqref{-1c}, we arrive at
		\[
		\int_C |F(z)|^p(1-|z|^2)^{1+\alpha}|dz|
		\leq
		D_{C,\alpha}\,c_{p,\alpha}^p
		\int_{\mathbb D}|f(z)|^p\,dA_\alpha(z).
		\]
		Since $u=\Re(F)$, we have $|u(z)|^p\leq |F(z)|^p$, and consequently,
		\[
		\int_C |u(z)|^p(1-|z|^2)^{1+\alpha}|dz|
		\leq
		D_{C,\alpha}\,c_{p,\alpha}^p
		\int_{\mathbb D}|f(z)|^p\,dA_\alpha(z).
		\]
		Next, we apply the same argument to the harmonic function $if$. Since the real part of $if$ is $-v$, it reveals that
		\[
		\int_C |v(z)|^p(1-|z|^2)^{1+\alpha}|dz|
		\leq
		D_{C,\alpha}\,c_{p,\alpha}^p
		\int_{\mathbb D}|f(z)|^p\,dA_\alpha(z).
		\]
		
		We now make use of the following identity
		\[
		|f(z)|^p=\left(u(z)^2+v(z)^2\right)^{p/2}
		\]
		and proceed by considering the following two cases.
		
		\medskip
		
		\noindent
		\textbf{Case 1.} $p\geq2$.
		Since
		\[
		(a+b)^{p/2}
		\leq
		2^{\frac{p}{2}-1}
		\left(a^{p/2}+b^{p/2}\right),
		\qquad a,b\geq0,
		\]
		we obtain
		\[
		\left(u^2+v^2\right)^{p/2}
		\leq
		2^{\frac{p}{2}-1}
		\left(|u|^p+|v|^p\right).
		\]
		Hence,
		\[
		\begin{aligned}
			\int_C |f(z)|^p(1-|z|^2)^{1+\alpha}|dz|
			&\leq
			2^{\frac{p}{2}-1}
			\left(
			\int_C |u|^p(1-|z|^2)^{1+\alpha}|dz|
			+
			\int_C |v|^p(1-|z|^2)^{1+\alpha}|dz|
			\right)\\
			&\leq
			2^{\frac{p}{2}}
			D_{C,\alpha}\,c_{p,\alpha}^p
			\int_{\mathbb D}|f(z)|^p\,dA_\alpha(z).
		\end{aligned}
		\]
		
		\medskip
		
		\noindent
		\textbf{Case 2.} $0<p<2$.
		Since $0<p/2<1$,
		\[
		(u^2+v^2)^{p/2}
		\leq
		|u|^p+|v|^p.
		\]
		Therefore,
		\[
		\begin{aligned}
			\int_C |f(z)|^p(1-|z|^2)^{1+\alpha}|dz|
			&\leq
			\int_C |u|^p(1-|z|^2)^{1+\alpha}|dz|
			+
			\int_C |v|^p(1-|z|^2)^{1+\alpha}|dz|\\
			&\leq
			2D_{C,\alpha}\, c_{p,\alpha}^p
			\int_{\mathbb D}|f(z)|^p\,dA_\alpha(z).
		\end{aligned}
		\]
		Thus, in all cases,
		\[
		\int_C |f(z)|^p(1-|z|^2)^{1+\alpha}|dz|
		\leq
		K_{p,\alpha,C}
		\int_{\mathbb D}|f(z)|^p\,dA_\alpha(z),
		\]
		where
		\[
		K_{p,\alpha,C}
		=
		\begin{cases}
			2^{p/2}D_{C,\alpha}\, c_{p,\alpha}^p, & p\geq2,\\[2mm]
			2D_{C,\alpha}\, c_{p,\alpha}^p, & 0<p<2.
		\end{cases}
		\]
		This completes the proof.
	\end{pf}

	\begin{proof} [{\bf Proof of Theorem \ref{Theorem-2.3}}]
		Since the weight is constant on the circle $C_r$, we have
		\[
		(1-|z|^2)^\alpha=(1-r^2)^\alpha,\qquad z\in C_r.
		\]
		Consequently,
		\[
		\int_{C_r}|f(z)|^p(1-|z|^2)^{1+\alpha} |dz|
		=
		(1-r^2)^{1+\alpha}
		\int_{C_r}|f(z)|^p|dz|.
		\]
		Moreover,
		\[
		\int_{C_r}|f(z)|^p|dz|
		=
		2\pi r\,M_p^p(r,f).
		\]
		Since the integral means $M_p^p(r,f)$ are increasing with respect to $r$, it follows that
		\[
		\begin{aligned}
			\|f\|_{a_\alpha^p}^p
			&=
			(1+\alpha)\int_{\mathbb D}|f(z)|^p(1-|z|^2)^\alpha\,dA(z)\\
			&=
			2\, (1+\alpha)
			\int_0^1
			M_p^p(\rho,f)
			(1-\rho^2)^\alpha
			\rho\,d\rho\\
			&\geq
			2\, (1+\alpha)
			M_p^p(r,f)
			\int_r^1
			(1-\rho^2)^\alpha
			\rho\,d\rho.
		\end{aligned}
		\]
		A straightforward computation shows that
		\[
		\int_r^1
		(1-\rho^2)^\alpha
		\rho\,d\rho
		=
		\frac{(1-r^2)^{\alpha+1}}
		{2(\alpha+1)}.
		\]
		Substituting this into the preceding inequality yields
		\[
		M_p^p(r,f)
		\leq
		\frac{1}
		{(1-r^2)^{\alpha+1}}
		\|f\|_{a_\alpha^p}^p.
		\]
		Therefore,
		\[
		\begin{aligned}
			\int_{C_r}|f(z)|^p(1-|z|^2)^{\alpha+1}|dz|
			&= (1-r^2)^{1+\alpha}
			\int_{C_r}|f(z)|^p|dz|\\
			&=
			2\pi r(1-r^2)^{1+\alpha}
			M_p^p(r,f)\\
			&\leq
			2\pi r\,
			\|f\|_{a_\alpha^p}^p.
		\end{aligned}
		\]
		Recalling the definition of the weighted Bergman norm, we conclude that
		\[
		\int_{C_r}|f(z)|^p(1-|z|^2)^{1+\alpha}\,|dz|
		\leq 2\, \pi \,r(\alpha+1)\, 
		\int_{\mathbb D}
		|f(z)|^p(1-|z|^2)^\alpha\,dA(z).
		\]
		This establishes the asserted inequality and completes the proof.
	\end{proof}

	\begin{proof} [{\bf Proof of Lemma \ref{Lemma-2.1}}]
		For $0 < r < 1$, we define the dilated functions $h_r(z) = h(rz)$ and $g_r(z) = g(rz)$. Consequently, we define the dilated harmonic function $f_r(z) = h_r(z) + \overline{g_r(z)}$. Because $h_r$ and $g_r$ are analytic in a neighborhood containing the closed unit disk, $f_r$ has continuous boundary values and inherently belongs to the harmonic Hardy space $h^p$. Furthermore, the condition at the origin is preserved under dilation, since $h_r(0) = h(0)$ and $g_r(0) = g(0)$. Thus, we have:$$\Re(g_r(0)h_r(0)) = \Re(g(0)h(0)) \ge 0$$
		
		Applying \cite[Theorem 2.1]{Kalaj-2019}) to the boundary values of $f_r$ on the unit circle $z = e^{i\theta}$, we obtain
		$$\int_0^{2\pi} \left( |h_r(e^{i\theta})|^2 + |g_r(e^{i\theta})|^2 \right)^{p/2} \frac{d\theta}{2\pi} \le \frac{1}{(1-|\cos(\frac{\pi}{p})|)^{p/2}} \int_0^{2\pi} |h_r(e^{i\theta}) + \overline{g_r}(e^{i\theta})|^p \frac{d\theta}{2\pi}$$
		Multiply both sides of above equation by the radial weight $2r(1-r^2)^\alpha$ and integrate with respect to $r$ from $0$ to $1$, we have 
		\begin{align*}
			\int_0^1 \int_0^{2\pi} 2r(1-r^2)^\alpha & \left( |h_r(e^{i\theta})|^2 + |g_r(e^{i\theta})|^2 \right)^{p/2} \frac{d\theta}{2\pi}  dr   \\
			& \le\frac{1}{(1-|\cos(\frac{\pi}{p})|)^{p/2}} \int_0^1 \int_0^{2\pi} 2r(1-r^2)^\alpha |h_r(e^{i\theta}) + \overline{g_r}(e^{i\theta})|^p \frac{d\theta}{2\pi}  dr
		\end{align*}
		Substituting $z = re^{i\theta}$, the normalized weighted area measure is given by $dA_\alpha(z) = (1+\alpha)\frac{1}{\pi} (1-r^2)^\alpha r \, d\theta \, dr$. Transitioning from the iterated integrals to the area integral over $\mathbb{D}$ yields
		$$\int_{\mathbb{D}} (|h(z)|^2 + |g(z)|^2)^{p/2} dA_\alpha(z) \le \frac{1}{(1-|\cos(\frac{\pi}{p})|)^{p/2}} \int_{\mathbb{D}} |h(z) + \overline{g}(z)|^p dA_\alpha(z)$$Finally, taking the $1/p$-th power of both sides gives the desired inequality:$$\left( \int_{\mathbb{D}} (|h|^2 + |g|^2)^{p/2} dA_\alpha \right)^{1/p} \le \frac{1}{(1-|\cos(\pi/p)|)^{1/2}} \left( \int_{\mathbb{D}} |h+\bar{g}|^p dA_\alpha \right)^{1/p}$$This completes the proof.
	\end{proof}
	
	\begin{pf} [{\bf Proof of Theorem \ref{Theorem-2.4}}]
		Let$$L(f) = \int_{D_0+D_1} |f(z)|^p (1-|z|^2)^{1+\alpha} |dz|$$
		Suppose that $f = h + \bar{g}$ with $f(0)=g(0) = 0$. Since
		$ |f(z)| \le |h(z)| + |g(z)|$, it follows that
		\begin{equation}\label{1}
			L(f) \le \int_{D_0+D_1} (|h(z)| + |g(z)|)^p (1-|z|^2)^{1+\alpha} |dz|. 
		\end{equation}
		For $0 < r < 1$, define
		$$\Phi_r(z) = (|h(rz)| + |g(rz)|)^p.
		$$
		Since $h$ and $g$ are analytic in $\mathbb{D}$, the dilations $h_r$ and $g_r$ belong to $H^p$ for every $r \in (0,1)$.
		Hence, by \cite[Lemma 2]{Das-Kaliraj-2022},
		$$\int_{D_0+D_1} \Phi_r(z) |dz| \le \frac{1}{\sin\frac{\theta}{2} + \cos\frac{\theta}{2}} \int_{|z|=1} \Phi_r(z) |dz|.
		$$
		Writing the boundary integral in polar form yields
		$$\int_{D_0+D_1} (|h(rz)| + |g(rz)|)^p |dz| \le \frac{1}{\sin\frac{\theta}{2} + \cos\frac{\theta}{2}} \int_0^{2\pi} (|h(re^{it})| + |g(re^{it})|)^p dt.
		$$
		Multiplying both sides by $r(1-r^2)^\alpha$ and integrate over $0<r<1$, we obtain
		\begin{align}\label{2}
			&	\int_0^1 r(1-r^2)^\alpha  \int_{D_0+D_1} (|h(rz)| + |g(rz)|)^p |dz| dr \notag \\ 
			& \le \frac{1}{\sin\frac{\theta}{2} + \cos\frac{\theta}{2}} \int_0^1 r(1-r^2)^\alpha \int_0^{2\pi} (|h(re^{it})| + |g(re^{it})|)^p dt dr.
		\end{align} 
		Now set 
		$$I := \int_{D_0+D_1} \int_0^1 \left(|h(rz)| + |g(rz)|\right)^p r (1-r^2)^\alpha dr |dz|.
		$$
		Write
		$$D_j = \{t\xi_j : -1 \le t \le 1\}, \quad |\xi_j| = 1 \quad \text{for}\ j=0,1.
		$$
		Since $z=t\xi_j$ on $D_j$, we have $|dz|=dt$, and therefore
		$$I = \sum_{j=0}^1 \int_{t=-1}^1 \int_{r=0}^1 (|h(rt\xi_j)| + |g(rt\xi_j)|)^p (1-r^2)^\alpha r \,dr\, dt .
		$$
		For each fixed $t\neq 0$, let $\rho=r\,t$. Then
		$r=\rho/t$ and $dr=d\rho/t$, so that 
		$$I = \sum_{j=0}^1 \int_{t=-1}^1 \int_{\rho=0}^t (|h(\rho\xi_j)| + |g(\rho\xi_j)|)^p \left(1-\frac{\rho^2}{t^2}\right)^\alpha \frac{\rho }{t^2}d\rho dt.
		$$
		Writing $F(\rho\xi_j) = (|h(\rho\xi_j)| + |g(\rho\xi_j)|)^p$ and applying Fubini's theorem, we deduce that
		\begin{align*}
			I &= \sum_{j=0}^1 \int_{\rho=0}^1 \int_{t=\rho}^1 F(\rho\xi_j)  \left(1-\frac{\rho^2}{t^2}\right)^\alpha  \frac{\rho}{t^2} dt d\rho \\
			&+ \sum_{j=0}^1 \int_{\rho=0}^{-1} \int_{t=-1}^\rho F(\rho\xi_j) \left(1-\frac{\rho^2}{t^2}\right)^\alpha \frac{\rho}{t^2} dt d\rho.
		\end{align*}
		For the first inner integral, let $u = \frac{\rho}{t}$. 	Then the integral $$\int_{t=\rho}^1 \left(1-\frac{\rho^2}{t^2}\right)^\alpha \frac{\rho }{t^2}dt$$ 
		becomes
		$$ \int_{u=\rho}^1 (1-u^2)^\alpha du := \kappa_\alpha(\rho).$$
		Similarly, the integral $$\int_{t=-1}^\rho \left(1-\frac{\rho^2}{t^2}\right)^\alpha \frac{\rho }{t^2}dt$$ becomes
		$$ 
		- \int_{u=-\rho}^1 (1-u^2)^\alpha du = -\kappa_\alpha(-\rho).
		$$
		Consequently,
		\begin{align*}
			I &= \sum_{j=0}^1 \int_{\rho=0}^1 F(\rho\xi_j) \kappa_\alpha(\rho) d\rho+ \sum_{j=0}^1 \int_{\rho=0}^{-1} F(\rho\xi_j) (-\kappa_\alpha(-\rho)) d\rho \\
			& = \sum_{j=0}^1 \int_{\rho=0}^1 F(\rho\xi_j) \kappa_\alpha(\rho) d\rho+ \sum_{j=0}^1 \int_{\rho=-1}^0 F(\rho\xi_j) \kappa_\alpha(-\rho) d\rho.
		\end{align*}
		
		\noindent \textbf{Estimate of $\kappa_\alpha(\rho)$:} For $\rho \in [0,1]$, we have
		\[
		(1-u^2)^\alpha \ge (1-u)^\alpha,\qquad u\in[\rho,1],\ \alpha>0.
		\]
		Therefore, $$\kappa_\alpha(\rho) \ge \int_\rho^1 (1-u)^\alpha du = \frac{(1-\rho)^{\alpha+1}}{1+\alpha}.
		$$
		Since $1+\rho\le2$, we obtain
		$$ 1-\rho = \frac{1-\rho^2}{1+\rho} \ge \frac{1-\rho^2}{2} ,
		$$
		and hence
		$$ (1-\rho)^{\alpha+1} \ge \frac{1}{2^{1+\alpha}} (1-\rho^2)^{1+\alpha}.
		$$
		Thus, $$ \kappa_\alpha(\rho) \ge \frac{(1-\rho^2)^{1+\alpha}}{(1+\alpha)2^{1+\alpha}}.$$
		
		\noindent \textbf{Estimate of $\kappa_\alpha(-\rho)$:} For $\rho \in [-1, 0]$,$$\kappa_\alpha(-\rho) = \int_{u=-\rho}^1 (1-u^2)^\alpha du.$$
		Setting $s=-\rho$, we have $s\in[0,1]$ and $$\kappa_\alpha(s) = \int_{u=s}^1 (1-u^2)^\alpha du.
		$$ 
		Applying the estimate established above yields
		$$\kappa_\alpha(s) \ge \frac{(1-s^2)^{1+\alpha}}{(1+\alpha)2^{1+\alpha}}.
		$$
		Hence, for every $\rho\in[-1,1]$,
		$$\kappa_\alpha(-\rho) \ge \frac{(1-\rho^2)^{1+\alpha}}{(1+\alpha)2^{1+\alpha}}.$$
		Therefore,
		\begin{align}\label{3}
			I \ge  &\frac{1}{(1+\alpha)2^{1+\alpha}} \sum_{j=0}^1 \left( \int_{\rho=0}^1 F(\rho\xi_j) (1-\rho^2)^{1+\alpha} d\rho + \int_{\rho=-1}^0 F(\rho\xi_j) (1-\rho^2)^{1+\alpha} d\rho \right) \notag \\ 
			&= \frac{1}{(1+\alpha)2^{1+\alpha}} \sum_{j=0}^1 \int_{\rho=-1}^1 F(\rho\xi_j) (1-\rho^2)^{1+\alpha} d\rho \notag \\ 
			&= \frac{1}{(1+\alpha)2^{1+\alpha}} \int_{D_0+D_1} F(z) (1-|z|^2)^{1+\alpha} |dz| \notag\\ 
			&= \frac{1}{(1+\alpha)2^{1+\alpha}} \int_{D_0+D_1} (|h(z)| + |g(z)|)^p (1-|z|^2)^{1+\alpha} |dz| \notag \\ 
			& \geq \frac{1}{(1+\alpha)2^{1+\alpha}} L(f).
		\end{align}
		On the other hand, the right-hand side of \eqref{2} is
		$$\frac{\pi}{(1+\alpha)\left(\sin\frac{\theta}{2} + \cos\frac{\theta}{2}\right)} \int_{\mathbb{D}} (|h(z)| + |g(z)|)^p dA_\alpha(z).$$
		Since for $p>1$, $(|a|+|b|)^p\le 2^{p/2}(|a|^2+|b|^2)^{p/2}$, it follows from Lemma \ref{Lemma-2.1} that
		\begin{align}\label{4}
			\int_{\mathbb{D}}  (|h(z)| + |g(z)|)^p dA_\alpha(z) 
			\le  \frac{2^{p/2}}{(1-|\cos(\frac{\pi}{p})|)^{p/2}} \int_{\mathbb{D}} |h(z) + \overline{g}(z)|^p dA_\alpha(z).
		\end{align}
		Combining \eqref{1}, \eqref{2}, \eqref{3}, and \eqref{4}, we conclude that
		$$L(f) \le \frac{2^{1+\alpha}\, \pi}{\sin\frac{\theta}{2} + \cos\frac{\theta}{2}} \frac{2^{p/2}}{(1-|\cos(\frac{\pi}{p})|)^{p/2}} \|f\|^p_{a_{\alpha}^p}.$$
		This proves the theorem.
	\end{pf}

	\noindent\textbf{Compliance of Ethical Standards:}\\
	\noindent\textbf{Conflict of interest.} The authors declare that there is no conflict  of interest regarding the publication of this paper.
	\vspace{1.5mm}
	
	\noindent\textbf{Data availability statement.}  Data sharing is not applicable to this article as no datasets were generated or analyzed during the current study.\vspace{1.5mm}
	
	\noindent\textbf{Authors contributions.} The authors are listed in alphabetical order by surname. Both authors made equal contributions to the research, the derivation of the results, and the writing and preparation of the manuscript.
	
\end{document}